\documentclass[11pt]{article}
\usepackage{caption}
\usepackage{epsf,epsfig,amsfonts,amsgen,amsmath,amstext,amsbsy,amsopn,amsthm
}
\usepackage{ebezier,eepic}
\usepackage{color}
\usepackage{multirow}
\usepackage{mathrsfs}
\usepackage{tikz}
\usepackage{enumerate}
\usepackage{enumitem}
\usepackage{url}
\usepackage{pgf,tikz,pgfplots}
\usepackage{mathrsfs}
\usepackage{amssymb}

\usetikzlibrary{arrows}

\newtheorem{dfn}{Definition}[section]
\newtheorem{thm}[dfn]{Theorem}
\newtheorem{lem}[dfn]{Lemma}

\newtheorem{corollary}[dfn]{Corollary}
\newtheorem{conjecture}[dfn]{Conjecture}

\def\ex{\mathrm{ex}}

\newcommand{\1}{{\uppercase\expandafter{\romannumeral1}}}
\newcommand{\2}{{\uppercase\expandafter{\romannumeral2}}}

\let\svthefootnote\thefootnote
\newcommand\blankfootnote[1]{%
	\let\thefootnote\relax\footnotetext{#1}%
	\let\thefootnote\svthefootnote%
}

\begin{document}

\title{Upper bounds on the extremal number of the 4-cycle}

\author{
Jie Ma
~~~~~~~
Tianchi Yang
}

\date{}

\blankfootnote{School of Mathematical Sciences, University of Science and Technology of China, Hefei, Anhui 230026, China.
Partially supported by the National Key R and D Program of China 2020YFA0713100,
National Natural Science Foundation of China grant 12125106, and Anhui Initiative in Quantum Information Technologies grant AHY150200.
Emails: jiema@ustc.edu.cn, ytc@mail.ustc.edu.cn.}

\maketitle
\begin{abstract}
We obtain some new upper bounds on the maximum number $f(n)$ of edges in $n$-vertex graphs without containing cycles of length four.
This leads to an asymptotically optimal bound on $f(n)$ for a broad range of integers $n$ as well as
a disproof of a conjecture of Erd\H{o}s from 1970s which asserts that $f(n)=\frac12 n^{3/2}+\frac14 n+o(n)$.
\end{abstract}

\section{Introduction}
Let $\ex(n,C_4)$ denote the maximum number of edges in an $n$-vertex {\it $C_4$-free} graph.\footnote{Throughout this paper, a graph is called {\it $C_4$-free} if it does not contain a cycle of length four as a subgraph.}
The study of this extremal number can be dated back to Erd\H{o}s \cite{E38} eighty years ago and has a rich, lasting influence on the development of extremal graph theory (see \cite{FS}).

It is well-known (see K\H{o}v\'ari-S\'os-Tur\'an \cite{KST} and Reiman \cite{R58}) that for any positive integer $n$,
\begin{equation}\label{equ:Reiman}
\ex(n,C_4)\le \frac{n}{4}(1+\sqrt{4n-3})=\frac12 n^{3/2}+\frac{n}{4}-O(n^{1/2}).
\end{equation}
Notably, the well-known friendship theorem of Erd\H{o}s-R\'enyi-S\'os \cite{ERS} shows that an equation can never hold in \eqref{equ:Reiman}.
On the other hand, using polarity graphs defined by finite projective planes,
Brown \cite{Br66} and Erd\H{o}s-R\'enyi-S\'os \cite{ERS} independently proved the following famous lower bound that
\begin{equation}\label{equ:Brown}
\ex(q^2+q+1,C_4)\geq \frac12q(q+1)^2 \mbox{ for all prime powers } q.
\end{equation}
These two results together with some basic property on the distribution of prime numbers imply the asymptotic formula that $\ex(n,C_4)=(\frac 12+o(1) )n^{3/2}$.

Motivated by above results, Erd\H{o}s raised several conjectures in 1970s to enrich the understanding on the extremal number $\ex(n,C_4)$.
In \cite{E76} Erd\H{o}s conjectured that the lower bound \eqref{equ:Brown} is sharp for prime powers $q$.
This was fully resolved by F\"uredi in \cite{Fur83,Fur96}, where he proved that
\begin{equation}\label{equ:Furedi}
\ex(q^2+q+1,C_4)\leq \frac12q(q+1)^2 \mbox{ for all integers } q\geq 14.
\end{equation}
Results related to \eqref{equ:Furedi} can be found in \cite{FKNW,HMY};
it is also worth noting that the bipartite version of the extremal number of $C_4$ has been well studied in \cite{DHS} (see Section 3.1 of \cite{FS}).
If one substitutes $n=q^2+q+1$ for prime powers $q$ in \eqref{equ:Brown},
then it yields that $\ex(n, C_4)\geq \frac12 n^{3/2}+\frac{n}{4}-O(n^{1/2})$ for such integers $n$.
Note that this meets the upper bound \eqref{equ:Reiman} up to the error term $O(n^{1/2})$.
Erd\H{o}s \cite{E76,E78} made the following tempting conjecture.

\begin{conjecture}[Erd\H{o}s \cite{E76,E78}]\label{conj: Erdos}
It holds that $$\ex(n,C_4)=\frac 12 n^{3/2}+\frac14 n+o(n).$$
\end{conjecture}

\noindent He also commented in \cite{E78} that ``it is not impossible that the error term is $O(n^{1/2})$.''

The present paper aims to establish new upper bounds on $\ex(n,C_4)$ (see Theorems~\ref{thm: q^2+q+1-r} and \ref{thm: q^2+q+1+r} below).
Contrary to the supportive evidences, these bounds imply the following result,
which shows that the above conjecture of Erd\H{o}s does not hold in a strong sense.

\begin{thm}\label{thm:main}
There exist some real $\epsilon>0$ and a positive density of integers $n$ such that $$\ex(n,C_4)\leq \frac12 n^{3/2}+\left(\frac14-\epsilon\right)n.$$
\end{thm}

This will follow by our upper bounds on $\ex(n, C_4)$.
Our computation shows that $\epsilon$ can be taken as any positive real less than $0.075$.
This number is unlikely to be tight,
so we did not try to optimize our calculation as well as the constants appearing in the forthcoming results;
see the concluding remarks for more discussion on the real $\epsilon$.
To proceed, we introduce some notation.
For integers $q\geq 0$, let $I_q$ denote the set of $2q+1$ consecutive integers $\{q^2+1,..., (q+1)^2\}$;
let $I_q^-=\{q^2+1,...,q^2+q\}$ and $I_q^+= \{q^2+q+2,..., (q+1)^2\}$ so that $I_q=I_q^-\cup \{q^2+q+1\}\cup I_q^+$.
Note that these $I_q$'s form a partition of the set of positive integers.

Our first result on the upper bound of $\ex(n, C_4)$ focuses on integers $n$ from the sets $I_q^-$.

\begin{thm}\label{thm: q^2+q+1-r}
Let $n=q^2+q+1-r$ be an integer in $I_q^-$ such that $r\leq 0.01q$ is sufficiently large.
Then $$\ex(n,C_4)=\ex(q^2+q+1-r,C_4)\le \frac 12q(q+1)^2-0.92rq.$$
\end{thm}

We point out that this bound can be further improved to the form \eqref{equ:tight-eps} (see the remark after the proof of Theorem~\ref{thm: q^2+q+1-r} in Section~3).
As a corollary, this yields the following asymptotic bound.

\begin{corollary}\label{coro:asy-bound}
Let $q$ be a prime power and $r=o(q)$ be sufficiently large. Then $$\ex(q^2+q+1-r,C_4)=  \frac 12q(q+1)^2-(r+o(1))q.$$
\end{corollary}

Our second result on the upper bound of $\ex(n, C_4)$ considers integers $n$ belonging to the sets $I_q^+$.

\begin{thm}\label{thm: q^2+q+1+r}
Let $n=q^2+q+1+r$ be a sufficiently large integer in $I_q^+$ such that $r\leq 0.6q$. Then
$$\ex(n,C_4)=\ex(q^2+q+1+r,C_4)\le \frac 12\big(q^2+q+1+\max\{r,2r-0.3q\}\big)(q+1).$$
\end{thm}

\begin{figure}[h]
\centering
 \vspace{-0.5cm}
 \hspace{-1cm}
\begin{minipage} {.45 \textwidth}
\begin{tikzpicture}[scale=0.8]
\tikzstyle{every node}=[font=\small,scale=0.7]
\begin{axis}
[axis lines = left,
hide axis,
xmin=0,
ymin=0,
line width=0.6pt,
]
\addplot[domain=300:2000,
samples= 100,
color=black,]{(1+sqrt(4*x-3))*x/4};
\end{axis}
\node at (0,-0.3) {0};
\node at (2.7,6) {\large $F(n)=n(1+\sqrt{4n-3})/4$};
\node at (3.31,5.35) {\large $\approx n^{3/2}/2+n/4+o(n)$};
 \draw [-latex] (0,0)--(8 ,0) node[right  ] { $n$};
\draw [-latex] (0,0)--(0,6) ;


\draw [fill=black] (5 ,3.54 ) circle (1 pt) node[above=2mm] {\large $B$};
\draw [fill=black] (4.5 ,3.03 ) circle (1 pt) node[left=2mm] {\large$A$};
    \draw [dashed] (5 ,3.54 )--(5,0);
\draw [dashed] (4.5 ,3.03 )--(4.5,0);
\draw (3.9,0)--++(90:0)   node[below=0.3mm]{\small $ q^2-q+1$};
\draw (5.5,0)--++(90:0 )   node[below=0.3mm]{\small $ q^2+q+1$};

\draw [color=red,dashed ] (4.3 ,2.8 )  rectangle (5.2 ,3.7 ) ;
\draw[-latex,color=red,dashed ] (5.2,3.2) to [in=-150,out=-30] (9.9,3.05);
 \end{tikzpicture}
\end{minipage}
  \vspace{-0.3cm}
\hspace{-1cm}
\begin{minipage}[c]{.45\textwidth}
\begin{tikzpicture}[scale=0.7]
\tikzstyle{every node}=[font=\small,scale=0.7]
\clip(-2.5,-2) rectangle (11,8.5);

\draw [-latex] (5.5,0)--(9,0) node[right  ] { $n$};
\draw (0,0)--(4.5,0);
\filldraw (5 ,0) circle (1pt);\filldraw (4.8,0) circle (1pt);\filldraw (5.2,0) circle (1pt);
\draw [-latex] (0,0)--(0,6.9)  ;



\draw[color=black] (3.5, 4) node { line 1 };
\draw (0,0.5)--(8,6.5);
\draw [fill=black] (0,0.5) circle (1.5 pt)  node[left ]{\large $ A$};
\draw [fill=black] (8,6.5) circle (1.5 pt) node[above]{\large $B$};
\node at (0,-0.3) {$q^2-q+1$};

\draw[color=black] (7.8, 5.5) node {line 2};
\draw [color=red](8,6.5)--(7 ,5.35 );
\draw [fill=black] (7 ,5.35 ) circle (1 pt);
\draw [dashed] (7,5.35 )--(7,0 );
\draw [dashed] (8,6.5)--(8,0 );
\draw (7,0)--++(90:0.1)  node[below=1mm]{\small $ q^2+0.99q$};
\draw[color=black] (1.1 , 0.6) node {line 3};
\draw [color=yellow](0,0.5)--(2 ,1.2  );

\draw [dashed] (2 ,1.2  )--(2,0);
\draw (2, 0)--++(90:0.1)  node[below=1mm]{\small $ q^2-0.7q $};
\draw [color=black] (3.1 , 1.7) node {line 4};
\draw [color=blue](2,1.2)--(3.6,3.1 );
\draw [fill=black] (3.6,3.1 ) circle (1 pt);
\draw [dashed] (3.6,3.1 )--(3.6,0);
\node at  (3.8,-0.35){\small $ q^2-0.4q $};
\draw [fill=black] (2,1.2) circle (1 pt);
\draw (0,-1)--(0,-1.5);
\draw (5,-1)--(5,-1.5);
\draw (8,-1)--(8,-1.5);
\draw[-latex] (2 ,-1.25)--(0,-1.25);
\draw[-latex] (3 ,-1.25)--(5 ,-1.25);
\node at (2.5 ,-1.25){$I_{q-1}^+$};
\draw[-latex] (6 ,-1.25)--(5 ,-1.25);
\draw[-latex] (7 ,-1.25)--(8,-1.25);
\node at (6.5,-1.25){$I_{q }^-$};
\end{tikzpicture}
\end{minipage}
\caption{The old and new upper bounds on $\ex(n,C_4)$.}
\end{figure}
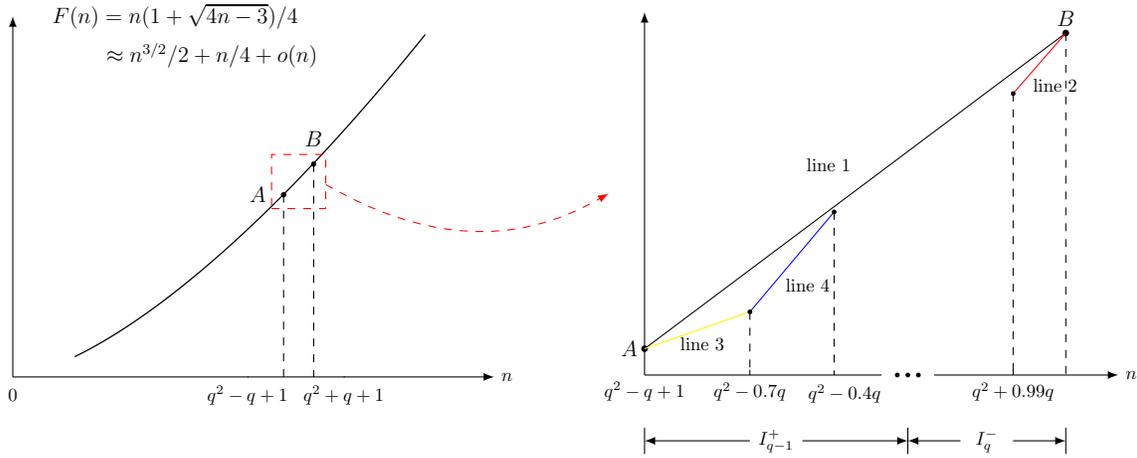

See Figure~1 for an illustration of the improvements on $\ex(n,C_4)$ in Theorems~\ref{thm: q^2+q+1-r} and \ref{thm: q^2+q+1+r}.\footnote{The curve $F(n)=\frac{n}{4}(1+\sqrt{4n-3})$ in the interval $[q^2-q+1,q^2+q+1]$ is quite close to a straight line with slope $0.75q$ (i.e., the line 1).
The upper bound of Theorem~\ref{thm: q^2+q+1-r} is indicated by the line 2 with some slope close to $q$,
while the upper bounds of Theorem~\ref{thm: q^2+q+1+r} are indicated by the line 3 and line 4 with slopes $0.5q$ and $q$, respectively.}

The rest of the paper is structured as follows.
In Section~2, we use Theorems \ref{thm: q^2+q+1-r} and \ref{thm: q^2+q+1+r} to prove Theorem \ref{thm:main}.
In Section~3, we introduce the key terminology, establish two useful lemmas and give a sketch for the proofs of the upper bounds on $\ex(n,C_4)$.
In Section~4, we show Theorem \ref{thm: q^2+q+1-r} and Corollary \ref{coro:asy-bound}.
In Section~5, we finish the proof of Theorem \ref{thm: q^2+q+1+r}.
Finally, in Section~6, we make some concluding remarks.

\section{Proof of Theorem \ref{thm:main}}
Assuming Theorems \ref{thm: q^2+q+1-r} and \ref{thm: q^2+q+1+r}, we now present a short proof of Theorem \ref{thm:main}.
(In fact, either Theorem \ref{thm: q^2+q+1-r} or Theorem \ref{thm: q^2+q+1+r} suffices to derive the statement of Theorem \ref{thm:main}.)

\begin{proof}[\bf Proof of Theorem \ref{thm:main}]
Let $\epsilon>0$ be some small absolute constant (i.e., $\epsilon=0.001$ suffices).
Let $n=q^2+q+1+r\in I_q$. So we have $-q \leq r\leq q$.
By some straightforward calculations, it holds that
\begin{equation}\label{equ:form}
\frac12 n^{3/2}+\frac14n=\frac 12q(q+1)^2+\frac34 rq+O(q).
\end{equation}

Let $N_1=\{q^2+q+1-r: 6\epsilon\leq r/q\leq 0.01\}$.
Consider any sufficiently large integer $n=q^2+q+1-r\in N_1$.
So $q$ and (thus) $r$ are sufficiently large as well.
We claim that every such $n\in N_1$ satisfies that $\ex(n,C_4)\leq \frac12 n^{3/2}+\left(\frac14-\epsilon\right)n.$
To see this, using Theorem \ref{thm: q^2+q+1-r}, we can derive that
$$\ex(n,C_4)\leq  \frac 12q(q+1)^2-0.92rq=\left(\frac12 n^{3/2}+\frac14n\right)-0.17rq+O(q)\leq \frac12 n^{3/2}+\left(\frac14-\epsilon\right)n,$$
where the equation holds by \eqref{equ:form} and the last inequality holds since $r\geq 6\epsilon q$ and $q$ is sufficiently large.
We note that the set $N_1$ has a positive density $\lim_{n\to \infty} \frac{|N_1\cap [n]|}{n}=(0.01-6\epsilon)/2>0$.

Consider another set of integers $N_2=\{q^2+q+1+r: 5\epsilon \leq r/q\leq 0.3\}$, which has a positive density $(0.3-5\epsilon)/2>0$.
We claim that any sufficiently large $n\in N_2$ also has $\ex(n,C_4)\leq \frac12 n^{3/2}+\left(\frac14-\epsilon\right)n.$
Indeed, by Theorem \ref{thm: q^2+q+1+r} and the above inequality \eqref{equ:form}, we have
$$\ex(n,C_4)\leq \frac 12(q^2+q+1+r)(q+1)=\left(\frac12 n^{3/2}+\frac14n\right)-\frac14rq+O(q)\leq \frac12 n^{3/2}+\left(\frac14-\epsilon\right)n,$$
where the last inequality holds because $r\geq 5\epsilon q$.
This completes the proof of Theorem \ref{thm:main}.
\end{proof}

\section{Preliminaries}
First we introduce the key notation on $C_4$-free graphs $G$ for the coming proofs.

\begin{dfn}\label{dfn:G}
Let $n\in I_q$ and $G$ be an $n$-vertex $C_4$-free graph.
As usual, we write $d(v)$ for the degree of a vertex $v$ in $G$.
The deficiency of a vertex $v$ in $G$ is defined by $$f(v)=q+1-d(v).$$
The deficiency of a subset $A\subseteq V(G)$ is given by $f(A)=\sum_{v\in A}f(v)$.
For each integer $i\geq 0$, we denote $S_i$ by the set of vertices of degree $i$ in $G$.
Finally, we let $S=\cup_{i\leq q} S_i$ and $S^+ =\cup_{j\geq q+2} S_j$.
\end{dfn}

Note that $V(G)=S\cup S_{q+1}\cup S^+$, and a vertex $v$ is in $S$ (or in $S^+$) if and only if it has positive (or negative) deficiency.
Next we prove some estimations on vertex degrees in $C_4$-free graphs.

\begin{lem}\label{equ: f(N(u))}
Let $n\in I_q$ and $G$ be an $n$-vertex $C_4$-free graph.
If we write $N(v)$ for the neighborhood of a vertex $v$ in $G$,
then we have $f(N(v))\geq q d(v)-n+1$.
\end{lem}

\begin{proof}
As $G$ is $C_4$-free,
all subsets $N(v_i)\backslash \{v\}$ for $v_i\in N(v)$ are disjoint.
So we have
$\sum_{v_i\in N(v)} (d(v_i)-1)=\sum_{v_i\in N(v)}  |N(v_i)\backslash \{v\}| \leq n-1.$
This implies that
\begin{equation*}
 f(N(v))=\sum_{v_i\in N(v)} (q+1-d(v_i))=qd(v)-\sum_{v_i\in N(v)} (d(v_i)-1)\ge qd(v)-n+1,
\end{equation*}
as desired.
\end{proof}

The coming lemma provides a crucial technical tool for later sections. Its proof idea is rooted in a lemma of \cite{Fur83}.

\begin{lem}\label{lem:counting 2-path_2}
Let $G$ be an $n$-vertex $C_4$-free graph with $m$ edges. Let $v_i$'s for $i\in [n]$ be vertices of $G$ of degree $d_i$.
If there exists some vertex $v$ such that $\{v_i\}_{i\in I}\subseteq N(v)$ for $I\subseteq [n]$ with $|I|=k$, then
\begin{equation*}
 \binom{n-\sum_{i\in I} d_i+k-1}{2}
\ge (n-k)\binom{\frac{2m-\sum_{i\in I}d_i+(k-1)d(v)-nk+ k}{n-k}}{2}.
\end{equation*}
\end{lem}
\begin{proof}
We say $v_iv_jv_\ell$ is a 2-path if $v_iv_j,v_jv_\ell\in E(G)$.
Let $X=\bigcup_{i\in I}N(v_i)$ and we count the number $M$ of 2-paths with both end-points in $V(G)\backslash X$.
Since $G$ is $C_4$-free, all $N(v_i)\backslash \{v\}$ are disjoint, implying that $|X|=\sum_{i\in I}( |N(v_i)|-1)+1=\sum_{i\in I} d_i-k+1$.
So it is evident that
$$M\le \binom{n-|X|}{2} = \binom{n-\sum_{i\in I} d_i+k-1}{2}.$$
On the other hand, for any $j\in [n]$ there are $\binom{|N(v_j)\backslash X|}{2}$ many counted 2-paths with the middle-point $v_j$.
For any $j\notin I$ we have $|N(v_j)\backslash X|\ge d_j-k$
and moreover, for $v_j\in N(v)\backslash \{v_i\}_{i\in I}$, we have $|N(v_j)\backslash X|= |N(v_j)\backslash \{v\} |= d_j-1$.
So it follows by Jensen's inequality that
\begin{align*}
M=\sum_{j\notin I}\binom{|N(v_j)\backslash X|}{2}
\geq \sum_{v_j\in \substack{N(v)\backslash \{v_i\}_{i\in I}}}\binom{d_j-1}{2} +\sum_{\substack{v_j\notin N(v)}} \binom{d_j-k}{2}\geq (n-k)\binom{\frac{L}{n-k}}{2},
\end{align*}
where $L=\sum_{v_j\in \substack{N(v)\backslash \{v_i\}_{i\in I}}}(d_j-1)+\sum_{\substack{v_j\notin N(v)}} (d_j-k)
=\sum_{j\notin I}d_j-(d(v)-k)-k(n-d(v)) =2m-\sum_{i\in I}d_i+(k-1)d(v)-nk+ k.$
Putting the above together, we finish the proof.
\end{proof}

Lastly, we would like to give a outline of the proof of Theorem~\ref{thm: q^2+q+1-r} (the proof of Theorem~\ref{thm: q^2+q+1+r} can be proved by a similar approach).
The proof is inspired by the work of F\"uredi \cite{Fur83,Fur96}.
There are several new ingredients as well.
The first ingredient comes from Lemma~\ref{lem:r0}, which roughly says that one only needs to consider $C_4$-free graphs with very large minimum degree.
Another ingredient is the deficiency $f(\cdot)$ in Definition~\ref{dfn:G}.\footnote{In literature where certain circumstances apply, the deficiency of a vertex $v$ is defined by $\max\{q+1-d(v),0\}$.}
Under this definition, we have $f(V(G))=O(rq)$ and we can maximize the benefits of this estimate.
Putting these ingredients together with the technical tool Lemma~\ref{lem:counting 2-path_2}, we prove in Lemma~\ref{lem:max-degree and f}
that the size of $S^+$ (i.e., the number of vertices whose degree exceeds $q+1$) can be bounded from above by $O(r^2)$.
This will eventually lead to a contradiction by assigning appropriate weights to the edges between the sets $S_{q+1}$ and $S$.

\section{Proof of Theorem \ref{thm: q^2+q+1-r}}
We first show that restricted to the range of integers considered in Theorem \ref{thm: q^2+q+1-r},
one may always assume that there exists some {\it extremal graph} with large minimum degree.\footnote{Throughout the rest of the paper,
a graph $G$ is called an \it{extremal graph} if it is $C_4$-free and has the maximum number $\ex(|V(G)|,C_4)$ of edges.}

\begin{lem}\label{lem:r0}
Let $q, r$ be integers satisfying $1\le r\le 0.3q$.
Assume that $\ex(q^2+q+1-r,C_4)\ge \frac 12 q(q+1)^2-\alpha rq $, where $0.2\leq \alpha \leq 1$.
Then there exists an integer $r_0\in [r, 3r]$ such that
$$\ex(q^2+q+1-r_0,C_4)\ge \frac 12 q(q+1)^2-\alpha r_0q$$
and
$$\ex(q^2+q+1-r_0,C_4)-\ex(q^2+q -r_0 ,C_4)\ge 0.2q.$$
In particular, the latter inequality shows that the minimum degree of any extremal graph on $q^2+q+1-r_0$ vertices is at least $0.2q$.
\end{lem}

\begin{proof}
Suppose for a contradiction that there does not exist such an integer $r_0\in [r,3r]$.
So we have $\ex(q^2+q+1-r, C_4)-\ex(q^2+q-r, C_4)<0.2q$.
Let $s$ be the largest integer in $[r+1, 3r]$ such that
\begin{equation}\label{equ: 0.2q}
\ex(q^2+q+1-x ,C_4)-\ex(q^2+q -x  ,C_4)< 0.2q  \text{ holds for all integers }   x\in [r,s-1].
\end{equation}
Then either (i) $s=3r$ or (ii) $\ex(q^2+q+1-s,C_4)-\ex(q^2+q -s,C_4)\ge 0.2q$.
Summing up \eqref{equ: 0.2q} for all integers $x\in [r,s-1]$, we obtain
\begin{equation}\label{equ:1-s}
\ex(q^2+q+1-s,C_4)>\ex(q^2+q+1-r ,C_4)-0.2(s-r)q\geq \frac 12 q(q+1)^2 -\alpha rq -0.2(s-r)q.
\end{equation}
As $\alpha\geq 0.2$, we see $\ex(q^2+q+1-s,C_4)\ge \frac 12 q(q+1)^2- \alpha sq.$
If (ii) occurs, then $s$ would be a desired integer.
So (i) occurs, i.e., $s=3r$.
Since $q\geq 4$ and $\alpha\leq 1\leq r$, we have $\frac34sq-\frac12(q+1)-\alpha rq-0.2(s-r)q=q\cdot(1.85r-\alpha r-1/2-1/q)>0$.
This inequality together with \eqref{equ:1-s} give that
\begin{equation}\label{equ:1-s'}
\ex(q^2+q+1-s,C_4)>\frac 12 (q+1)(q^2+q+1)-\frac 34 sq.
\end{equation}

Let $h(x)=\frac x4 (1+\sqrt{4x-3})$.
By \eqref{equ:Reiman}, we see that $\ex(n,C_4)\leq h(n)$ holds for all $n$.
Note that $h(q^2+q+1)=\frac 12 (q+1)(q^2+q+1)$ and $h'(x)=\frac{1+\sqrt{4x-3}}{4}+\frac{x}{2\sqrt{4x-3}}>\frac34 \sqrt{x}$ for $x\geq 1$.
So $h(n+1)-h(n)\geq h'(n)> \frac 34 \sqrt n\geq \frac 34 q$ holds for every integer $n\ge q^2+q+1-s\geq q^2$, where $s=3r\leq 0.9q$.
Adding up the above inequality for all integers $n$ in $[q^2+q+1-s,q^2+q]$,
we have
$$h(q^2+q+1-s)< h(q^2+q+1)-\frac 34 q\cdot s=\frac 12 (q+1)(q^2+q+1)-\frac 34sq< \ex(q^2+q+1-s,C_4),$$
where the last inequality is from \eqref{equ:1-s'}.
Clearly, this contradicts the bound $\ex(n,C_4)\leq h(n)$.
This final contradiction shows that the desired $r_0\in [r,3r]$ indeed exists, completing the proof.
\end{proof}

Recall Definition \ref{dfn:G}. In the next lemma we show that in $C_4$-free graphs with large minimum degree and sufficiently many edges, the size of $S^+$ can be bounded from above.

\begin{lem}\label{lem:max-degree and f}
Let $q$ be a sufficiently large integer and $r$ be an integer satisfying $0\le r\le 0.033q$.
Let $G$ be a $C_4$-free graph on $q^2+q+2-r$ vertices, with at least $\frac 12 q(q+1)^2-rq$ edges and of minimum degree at least $0.2q$.
Then $|S^+|\leq -f(S^+)\le 4r^2+16r+18$.
\end{lem}
\begin{proof}
First we claim that for every vertex $v$, $|N(v)\cap S^+|\leq 3r+8$.
Suppose not. Let $k=3r+9$. Then there exist vertices $v$ and $v_i$'s for $i\in [k]$ such that every $v_i\in N(v)\cap S^+$.
Let $d_i$ be the degree of $v_i$ and let $s=\sum_{i\in[k]} d_i$.
Since $s-k=\sum_{i\in[k]} |N(v_i)\backslash \{v\}|\le n-1$, we have $k(q+2)\le s \le n-1+k\le q^2+2q$.
Let $n=q^2+q+2-r$ and $m_0=\frac 12 q(q+1)^2- rq$.
By the proof of Appendix A, we get
\begin{equation}\label{equ: 3r+9}
F(q,r,s):=2(n-k)\binom{n-s+k-1}{2}-2(n-k)^2\binom{\frac{2m_0-s+(k-1) 0.2q -nk+ k}{n-k}}{2}<0.
\end{equation}
Since $e(G)\geq m_0$ and $d(v)\ge 0.2q$, using Lemma~\ref{lem:counting 2-path_2}, we can derive that
$$0>\frac{F(q,r,s)}{2(n-k)}\geq \binom{n-s+k-1}{2}-(n-k)\binom{\frac{2e(G)-s+(k-1)d(v) -nk+ k}{n-k}}{2}\geq 0,$$
a contradiction. This proves the claim that $|N(v)\cap S^+|\le 3r+8$ holds for every vertex $v$.

Next we consider the deficiency.
Let $u\in S^+$. Clearly $f(u)\le -1$. By Lemma~\ref{equ: f(N(u))}, we have
\begin{equation}\label{equ: fNu1}
\begin{split}
f(N(u))\ge qd(u)-n+1
=-qf(u)+r-1 \ge (q-1)(-f(u)).
\end{split}
\end{equation}
Now define a weight function $w$ on the edges $uv$ with $u\in S^+$ and $v\in S$ by assigning $w(uv)=f(v)$.
Let us count the total weight $W$ of these edges.
Every vertex $u\in S^+$ contributes at least $f(N(u)\cap S)\geq f(N(u))$,
so we can use \eqref{equ: fNu1} to obtain that
$$W\ge \sum_{u\in S^+} f(N(u))\ge (q-1)\sum_{u\in S^+}(-f(u))=(q-1)(-f(S^+)).$$
On the other hand, by the above claim, each vertex $v\in S$ has at most $3r+8$ neighbors in $S^+$,
so it contributes at most $3r+8$ times of its deficiency. Putting them together, we have
$$(q-1)(-f(S^+))\le W\le \sum_{v\in S}(3r+8)f(v)=(3r+8)f(S).$$
By the definition of $f$, we get $f(S)+f(S^+)=f(V(G))= n(q+1)-2e(G)\le (q^2+q+2-r)(q+1)-(q(q+1)^2-2rq)=(r+2)q-r+2$.
Therefore, we can derive
$$(q-1)(-f(S^+))\le (3r+8)f(S)\le (3r+8)\big((r+2)q-r+2-f(S^+)\big).$$
Further rearranging this inequality gives that
$$-f(S^+)\le \frac{(3r+8)\big((r+2)q-r+2 \big)}{q-3r-9}\leq \frac{(3r+8)(r+2)(q+1)}{0.9(q+1)}\leq 4r^2+16r+18,$$
where the last two inequalities hold because $0\le r\le0.033q$ and $q$ is large.
This finishes the proof.
\end{proof}

Now we are ready to prove Theorem \ref{thm: q^2+q+1-r}.

\begin{proof}[\bf Proof of Theorem \ref{thm: q^2+q+1-r}]
Let $q, r$ be sufficiently large integers satisfying $r\le 0.01q$.
Suppose for a contradiction that $\ex(q^2+q+1-r,C_4)\ge \frac 12 q(q+1)^2-\alpha rq$, where $\alpha=0.92$.
By Lemma \ref{lem:r0}, there exists some integer $r_0$ with $r\le r_0\le 3r\le 0.03q$ such that
any extremal graph $G$ on $q^2+q+1-r_0$ vertices has at least $\frac 12 q(q+1)^2-\alpha r_0 q$ edges and minimum degree at least $0.2q$.

Let $n=q^2+q+1-r_0$. By the definition of deficiency,
\begin{align*}
f(V(G))&=(q+1)n-2e(G)\leq (q+1)(q^2+q+1-r_0)- (q(q+1)^2-2\alpha r_0q)\\
&= (2\alpha -1)r_0q-r_0+q+1\leq (2\alpha -1)r_0q +q
\end{align*}
Applying Lemma \ref{lem:max-degree and f} to $G$, we have
$|S^+|\le -f(S^+)\le 4(r_0+1)^2+16(r_0+1)+18=4r_0^2+O(q).$
So
\begin{equation}\label{equ:S}
|S|\le f(S)=f(V(G))-f(S^+)\le (2\alpha -1)r_0q+4r_0^2+O(q).
\end{equation}
We also have
\begin{equation}\label{equ:S q+1}
|S_{q+1}|=(q^2+q+1-r_0)-|S^+|-|S|
\ge \left(1-(2\alpha -1)\frac{r_0}{q}-8\frac{r_0^2}{q^2}\right)q^2+O(q).
\end{equation}
For any vertex $u\in S_{q+1}$, by Lemma \ref{equ: f(N(u))} we can derive
\begin{align*}
 f(N(u))\geq qd(u)-n+1\geq q(q+1)-(q^2+q+1-r_0)+1=r_0.
\end{align*}

Now define a weight function $w$ on the edges $uv$ with $u\in S_{q+1}$ and $v\in S$  by assigning $w(uv)=f(v)$.
Let $W$ be the total weight of these edges.
Every vertex $u\in S_{q+1}$ contributes at least $f(N(u)\cap S)\geq f(N(u))$ to $W$,
while each vertex $v\in S$ contributes at most $f(v)d(v)\le f(v) q$. Thus
\begin{equation}\label{equ: Sq+1 f(S)}
|S_{q+1}|r_0\le \sum_{u\in S_{q+1}} f(N(u)) \le W\le \sum_{v\in S} f(v) q \le f(S)q.
\end{equation}
Using \eqref{equ:S} and \eqref{equ:S q+1}, we have that
\begin{align*}
f(S)q-|S_{q+1}|r_0
&\le \big((2\alpha -1)r_0q+4r_0^2\big)q-
\left(1-(2\alpha -1)\frac{r_0}{q}-8\frac{r_0^2}{q^2}\right)r_0q^2+O(q^2)\\
&=\left(8\frac{r_0^2}{q^2}+(2\alpha+3)\frac{r_0}{q}+2\alpha -2 \right)r_0q^2+o(r_0q^2),
\end{align*}
where $r_0\geq r$ is sufficiently large.
Recall that $\alpha=0.92$ and $r_0\le 0.03q$.
We see $F(\frac{r_0}{q})=8\frac{r_0^2}{q^2}+(2\alpha+3)\frac{r_0}{q}+2\alpha -2 $ is a quadratic function on $\frac{r_0}{q}$ with a negative axis of symmetry.
So $F(\frac{r_0}{q})\leq F(0.03)=-0.0076<0$.
Therefore $f(S)q-|S_{q+1}|r_0\leq (F(\frac{r_0}{q})+o(1))\cdot r_0q^2<0$, which contradicts to \eqref{equ: Sq+1 f(S)}.
This completes the proof of Theorem \ref{thm: q^2+q+1-r}.
\end{proof}

We would like to remark that the above proof can be modified to show: for any $\epsilon>0$,
\begin{equation}\label{equ:tight-eps}
\ex(q^2+q+1-r,C_4)\le \frac 12q(q+1)^2-(1-\epsilon)rq
\end{equation}
holds whenever $r/q=O(\epsilon)$ and $r=\Omega(1/\epsilon).$

To conclude this section, we now prove Corollary~\ref{coro:asy-bound} using the above inequality.

\begin{proof}[\bf Proof of Corollary~\ref{coro:asy-bound}]
Let $\epsilon>0$ be any real.
Let $q$ be a prime power and $r$ be an integer such that $r=O(\epsilon q)$ and $r=\Omega(1/\epsilon)$.
By results of \cite{Br66,ERS,Fur83,Fur96}, there exists some extremal graph $G$ on $q^2+q+1$ vertices with $\frac 12q(q+1)^2$ edges,
which has exactly $q+1$ vertices of degree $q$ and all other vertices of degree $q+1$.
Deleting any $r$ vertices of degree $q$ in $G$, one can obtain a $C_4$-free graph on $q^2+q+1-r$ vertices with at least $\frac 12q(q+1)^2-rq$ edges.
This together with \eqref{equ:tight-eps} show that $$\frac 12q(q+1)^2-rq\leq \ex(q^2+q+1-r,C_4)\leq \frac 12q(q+1)^2-(1-\epsilon)rq,$$
proving the corollary.
\end{proof}

\section{Proof of Theorem \ref{thm: q^2+q+1+r}}
Let $q$ be a sufficiently large integer.
We first prove that for any integer $r\in [1, 0.3q]$, it holds that
\begin{equation}\label{equ:r<0.3q}
\ex(q^2+q+1+r,C_4)\leq \frac 12(q^2+q+1+r)(q+1).
\end{equation}
Suppose for a contradiction that there exists some integer $r\in [1, 0.3q]$
satisfying $\ex(q^2+q+1+r,C_4)>\frac 12(q^2+q+1+r)(q+1)$.
Let $G$ be an extremal graph on $q^2+q+1+r$ vertices.
We will complete the proof of \eqref{equ:r<0.3q} by deriving a contradiction that $e(G)=\frac 12(q^2+q+1+r)(q+1).$

To see this, we first claim that every vertex $v$ in $G$ with degree $d(v)\ge 0.7q$ satisfies $|N(v)\cap S^+|< 0.55q$.
Suppose not. Let $k=0.55q$. Then there exist vertices $v$ and $v_i$'s for $i\in [k]$ such that every $v_i\in N(v)\cap S^+$.
Let $n=q^2+q+1+r$ and $s=\sum_{i\in[k]} d_i$.
Since $\sum_{i\in[k]} (d_i-1)=\sum_{i\in[k]} |N(v_i)\backslash \{v\}|\le n-1$, we know $k(q+2)\le s \le n-1+k\le q^2+2q$.
Let $m_0=\frac 12 n(q+1)$ so that $e(G)\ge m_0$.
By the calculations in Appendix B, we get
\begin{equation}\label{equ: 0.5q}
G(q,r,s):=2(n-k)\binom{n-s+k-1}{2}- 2(n-k)^2\binom{\frac{2m_0-s+(k-1)0.7q -nk+ k}{n-k}}{2}<0.
\end{equation}
Note that $d(v)\ge 0.7q$ and $e(G)\ge m_0$.
We can use Lemma~\ref{lem:counting 2-path_2} to derive that
$$0>\frac{G(q,r,s)}{2(n-k)}\geq \binom{n-s+k-1}{2}-(n-k)\binom{\frac{2e(G)-s+(k-1) d(v) -nk+ k}{n-k}}{2}\geq 0,$$
a contradiction.
So indeed, $|N(v)\cap S^+|< 0.55q$ holds for any vertex $v$ with  $d(v)\ge 0.7q$ in $G$.

Now consider the deficiency of vertices.
Any $u\in S^+$ has $f(u)\le -1$.
By Lemma \ref{equ: f(N(u))}, we have
\begin{equation}\label{equ: fNu}
\begin{split}
f(N(u))
&\ge qd(u)-n+1
=q(q+1-f(u))-(q^2+q+1+r)+1\\
&=-f(u)q-r
\ge (-f(u))\cdot(q-r).
\end{split}
\end{equation}

Define a weight function $w$ on the edges $uv$ with $u\in S^+$ and $v\in S$ by assigning $w(uv)=f(v)$.
Let $W$ be the total weight of these edges.
As every vertex $u\in S^+$ contributes at least $f(N(u)\cap S)\ge f(N(u))$, from \eqref{equ: fNu} we can get that
$$W\ge \sum_{u\in S^+} f(N(u))\ge (q-r)\sum_{u\in S^+}(-f(u))=(q-r)(-f(S^+)).$$
On the other hand, as we know $|N(v)\cap S^+|< 0.55q$ for vertices $v$ with $d(v)\ge 0.7q$,
we see that in fact every vertex $v\in S$ has at most $0.7q-1$ neighbors in $S^+$.
So it contributes at most $0.7q-1$ times of its deficiency.
Putting the above together, we get
$$(q-r)(-f(S^+))\le W\le \sum_{v\in S} ((0.7q-1) f(v))=(0.7q-1) f(S).$$
As $1\le r\le 0.3q$,
this implies that
\begin{equation}\label{equ:f(S^+)}
-f(S^+)\le f(S),
\end{equation}
where the equality holds if and only if $-f(S^+)= f(S)=0$.
By the definition of $f$, we have $f(S)+f(S^+)=f(V(G))= n(q+1)-2e(G)\le 0$.
That is, $f(S)\le -f(S^+)$.
In view of \eqref{equ:f(S^+)}, we can then derive $-f(S^+)= f(S)=0$, which says that $S=S^+=\emptyset$.
This also says that all vertices in $G$ have degree $q+1$,
so $e(G)=\frac 12 n(q+1)=\frac 12(q^2+q+1+r)(q+1)$.
However, this contradicts our assumption on the choice of $G$, thus completing the proof of \eqref{equ:r<0.3q}.

It remains to consider when $0.3q<r\leq 0.6q$.
In fact for any $n\in I_q^+$, by \eqref{equ:Reiman} that $\ex(n,C_4)\le \frac{n}{4}(1+\sqrt{4n-3})$,
we see that the minimum degree of any $n$-vertex extremal graph is at most $q+1$.
So it follows that $\ex(n,C_4)\le \ex(n-1,C_4)+(q+1)$ for any $n\in I_q^+$.
Then for $0.3q<r\leq 0.6q$, we can derive that
$$\ex(q^2+q+1+r,C_4)\le \ex(q^2+q+1+0.3q,C_4)+(r-0.3q)(q+1)\le \frac 12(q^2+q+1+2r-0.3)(q+1),$$
where the last inequality follows by \eqref{equ:r<0.3q}.
This finishes the proof of Theorem \ref{thm: q^2+q+1+r}.\qed

\section{Concluding remarks}
This paper mainly concerns upper bounds of the extremal number $\ex(n,C_4)$.
We prove that the upper bound estimations in Theorems~\ref{thm: q^2+q+1-r} and \ref{thm: q^2+q+1+r} hold for a positive proportion of integers in $I_q^-$ and $I_{q-1}^+$, respectively.
We actually speculate that the upper bound of Theorem~\ref{thm: q^2+q+1-r} and the restricted form of Theorem~\ref{thm: q^2+q+1+r} for $r\leq 0.3q$ (i.e., the line 2 and the line 3 in Figure~1)
hold for all but $o(q)$ integers in $I_q^-\cup I_{q-1}^+$.
This would be one step closer to the precise value of $\ex(n,C_4)$ and in particular, would imply that the real $\epsilon$ in Theorem~\ref{thm:main} can be chosen to be any positive real less than $1/4$
for a positive density of integers $n$.

The proof of Theorem~\ref{thm: q^2+q+1-r} indicates that it would be very helpful for estimating the extremal number if the high minimum degree condition can be provided.
However, we tend to believe that this approach does not work in general (we plan to explore this and related topics in a coming paper).

\bigskip

\bigskip

{\noindent \bf Acknowledgements}. The authors would like to thank Zolt\'an F\"uredi for helpful discussions and for bringing the reference \cite{DHS} to our attention.
The authors are also grateful to Boris Bukh for many kind suggestions.

\bibliographystyle{unsrt}

\section*{Appendices}
\subsection*{A. Justification of the inequality \eqref{equ: 3r+9}}
Here we present a detailed proof for the inequality \eqref{equ: 3r+9}.
Note that $n=q^2+q+2-r$, $k=9+3r$, $m_0=\frac 12 q(q+1)^2- rq$, $k(q+2)\le s\le q^2+2q$ and $0\le r\le 0.03q$.
Also we have
\begin{align*}
F(q,r,s )&= (n-s+k-1)(n-s+k-2)( n-k )\\
&- (2m_0-s+(k-1)0.2q-nk+ k) (2m_0-s+(k-1)0.2q-nk-n+2k).
\end{align*}
It suffices to show that under the conditions, namely, $q$ is sufficiently large, $k(q+2)\le s\le q^2+2q$ and $0\le r\le 0.03q$,
the maximum value $F_{\max}$ of $F(q,r,s)$ is less than $0$.
For fixed $q$ and $r$, $F$ is a quadratic function with variable $s$.
The coefficients of $s^2$ and $s$ are $q^2 + q - 4 r - 8$ and $-2 q^4 - 2 q^3 + q^2 (-2 r - 22) + q (-4.8 r - 18.8) + 22 r^2 +120 r + 122$, respectively.
So the axis of symmetry is $$\frac{-2 q^4 - 2 q^3-2q^2  r+o(q^3)}{-2(q^2 + q - 4 r - 8) }=q^2+o(q^2).$$ Since $k(q+2)\le s \le q^2+2q $, we have
\begin{align*}
 F_{\max}&= F(q,r,k(q+2) )=
q^4 (-0.2 r-0.2)+q^3 (6.6 r^2+50.6 r+104)
\\&+q^2 (-18 r^3-180.76 r^2-633.32 r-729.56)+
q (-51.6 r^3-487r^2-1567.2 r-1630.8)\\
&-9 r^4-76 r^3-243 r^2-656 r-1044=q^4 (-0.2 r-0.2)+  6.6q^3 r^2  -18q^2r^3+o(q^4(r+1))\\
&=r(-0.2q^4+6.6q^3 r-18q^2r^2)-0.2q^4+o(q^4(r+1)).
\end{align*}
In the range $0\le r\le 0.03q$, we have $-0.2q^4+6.6q^3 r-18q^2r^2\le -0.01q^4$, implying that $F_{\max}<0$. \qed

\subsection*{B. Justification of the inequality \eqref{equ: 0.5q}}
Recall that here we have $n=q^2+q+1+r$, $k=0.55q$ and $m_0=\frac 12 n(q+1)$.
Moreover, $1\le r\le 0.3q$ and $k(q+2)\le s\le q^2+2q$.
The function $G(q,r,s)$ can be rewrote as the following
\begin{align*}
G(q,r,s)&=(n-s+k-1)(n-s+k-2)( n-k )\\
&- (2m_0-s+(k-1)0.7q-nk+ k) (2m_0-s+(k-1)0.7q-nk-n+2k).
\end{align*}

Similar to Appendix A, it suffices to show that the maximum value $G_{\max}$ of $G(q,r,s)$ is less than zero,
subject to the restrictions that $1\le r\le 0.3q$ and $k(q+2)\le s\le q^2+2q$.

For fixed $q$ and $r$, $G$ is a quadratic function with the variable $s$.
The coefficients of $s^2$ and $s$ are $q^2 + 0.45q + r$ and $-2 q^4 - 3.1 q^3 -(4r-0.275) q^2   -(3.1r+0.5)  q  - 2 r^2 + 2$, respectively.
So the axis of symmetry is $$\frac{ 2 q^4 + 3.1 q^3  + 4 q^2 r +o(q^3)}{ 2(q^2 + 0.45 q + r) }=q^2+o(q^2).$$
Since $k(q+2)\le s \le q^2+2q $, we have
\begin{align*}
G_{\max}&= G(q,r,k(q+2))=-0.210375 q^5+(0.6975 r-0.206475)q^4 +(0.5535 r-0.342125)q^3
\\&+(1.6975 r^2-0.205   r-0.685)q^2 +(0.9 r^2-0.2  r-0.2)q +r^3-r \le  -0.21  q^5+  0.698 rq^4+o(q^5).
 \end{align*}
As $1\le r\le 0.3q$, we get $ -0.21  q^5+  0.698 rq^4\le-0.0006q^5$.
Therefore, we have $G_{\max}<0$ and thus complete the proof.
\qed
\end{document}